\documentclass[12pt]{article}
\usepackage{mathrsfs}
\usepackage{amsmath,amsfonts,amssymb,rotating,amsthm}
\usepackage{hyperref}
\usepackage{array}

\def\qed{\nopagebreak\hfill{\rule{4pt}{7pt}}}

\newtheorem{theo}{Theorem}[section]

\newtheorem{lemm}[theo]{Lemma}
\newtheorem{prop}[theo]{Proposition}
\newtheorem{coro}[theo]{Corollary}
\newtheorem{conj}[theo]{Conjecture}

\theoremstyle{remark}
\newtheorem{remark}[theo]{Remark}

\newdimen\Squaresize \Squaresize=11pt
\newdimen\Thickness \Thickness=0.7pt
\def\Square#1{\hbox{\vrule width \Thickness
   \vbox to \Squaresize{\hrule height \Thickness\vss
    \hbox to \Squaresize{\hss#1\hss}
   \vss\hrule height\Thickness}
\unskip\vrule width \Thickness} \kern-\Thickness}

\def\Vsquare#1{\vbox{\Square{$#1$}}\kern-\Thickness}

\def\moins{\raise 1pt\hbox{{$\scriptstyle -$}}}

\begin{document}

\begin{center}
{\large \bf On Ratio Monotonicity of a New Kind of Numbers Conjectured by Z.-W. Sun}
\end{center}

\begin{center}
Brian Y. Sun\\[8pt]
Center for Combinatorics, LPMC-TJKLC\\
Nankai University, Tianjin 300071, P. R. China\\[6pt]

Email: {\tt brian@mail.nankai.edu.cn}\\

\end{center}

\vspace{0.3cm} \noindent{\bf Abstract.} Recently, Z. W. Sun put forward a series of conjectures on monotonicity of combinatorial sequences in the form of $\{z_n/z_{n-1}\}_{n=N}^\infty$ and $\{\sqrt[n+1]{z_{n+1}}/\sqrt[n]{z_n}\}_{n=N}^\infty$ for some positive integer $N$, where $\{z_n\}_{n=0}^\infty$ is a sequence of positive integers. Luca and St\u{a}nic\u{a}, Hou et al., Chen et al., Sun and Yang proved some of them. In this paper, we give an affirmative answer to monotonicity of another new kind of number conjectured by Z. W. Sun via interlacing method for log-convexity and log-concavity of a sequence, and we also use the criterion for log-concavity of a sequence in the form of $\{\sqrt[n]{z_n}\}_{n=1}^\infty$ due to Xia.



\section{Introduction}
 Let $\{z_n\}_{n\geq 0}$ be a number-theoretic or combinatorial sequence of positive numbers. A sequence $\{z_n\}_{n\geq 0}$ of positive numbers is called \emph{(strictly) ratio monotonic} if the sequence $\{z_{n}/z_{n-1}\}_{n\geq 1}$ of its consecutive quotients  is (strictly) monotonic. The concept of ratio monotonicity is closely related log-convexity and log-concavity.
 A sequence $\{z_n\}_{n=0}^\infty$ is called \emph{log-convex(resp. log-concave)} if for all $n\geq 1$
\begin{equation}
\label{Ineq:log-cv and log-cc}
z_{n-1}z_{n+1}\geq z_n^2~~~(\text{resp.}z_{n-1}z_{n+1}\leq z_n^2).
\end{equation}
Correspondingly, if the inequality in \eqref{Ineq:log-cv and log-cc} is strict, we call the sequence $\{z_n\}_{n=0}^\infty$ is \emph{strictly log-convex(resp. log-concave)}.

Clearly, a sequence $\{z_n\}_{n=0}^\infty$ is (strictly) log-convex(resp. log-concave) if and only if the sequence $\{z_{n+1}/z_{n}\}_{n\geq 0}$ is (strictly) increasing(resp. decreasing). So, to study the ratio monotonicity is equivalent to study the log-convexity and log-concavity, see \cite{wz,zhu}. Up to now, the log-convex and log-concave sequences have been extensively investigated as they are often arise in combinatorics, algebra, geometry, analysis, probability and statistics, the reader can refer to \cite{brenti,wy,stanley} for details.

   Recently, Z. W. Sun \cite{sun1,sun2} posed  a series of conjectures about monotonicity of sequences of the following forms $\{z_{n+1}/z_n\}_{n\geq 0}^\infty$, $\{\sqrt[n]{z_n}\}_{n\geq 1}$ and $\{\sqrt[n+1]{z_{n+1}}/\sqrt[n]{z_{n}}\}_{n\geq 1}$.
It is worthy to mention that many scholars have made valuable progress on this subject, such as  Chen et al. \cite{cgw}, Hou et al. \cite{hsw}, Luca and St\u{a}nic\u{a} \cite{LS}, Wang an Zhu \cite{wz}, Sun and Yang \cite{sy} and Zhao \cite{zhao}, etc.

The main object of this paper is to prove a conjecture due to Z. W. Sun \cite{sun1} on ratio monotonicity of a sequence $\{R_n\}_{n=0}^\infty$.
This new kind of number is also introduced by him in \cite{sun1}.
It is defined in the following way:
\begin{equation}\label{Seq:R}
R_n=\sum_{k=0}^n\binom{n}{k}
\binom{n+k}{k}\frac{1}{2k-1},\,\,n=0,1,2,\ldots.
\end{equation}
Moreover, Z. W. Sun also defined the sequence $\{S_n\}_{n=0}^\infty$ in \cite{sun1}, where
\begin{equation}\label{Seq-S}
S_n=\sum_{k=0}^n\binom{n}{k}^2
\binom{2k}{k}(2k+1),\,\,n=0,1,2,\ldots.
\end{equation}

For the sequence $\{S_n\}_{n=0}^\infty$, Z. W. Sun conjectured that the sequence $\{S_n\}_{n=0}^\infty$  is strictly ratio increasing to the limit $9$ and the sequence $\{\sqrt[n]{S_n}\}_{n=1}^\infty$ is strictly ratio decreasing to the limit $1$, which are confirmed by Sun and Yang \cite{sy}.

\begin{theo}\emph{(\cite[Theorem 1.2 ]{sy})}\label{Thm:thm-S}
 The sequence $\{S_{n+1}/S_n\}_{n\geq 3}$ is strictly increasing to the limit $9$, and the sequence $\{\sqrt[n+1]{S_{n+1}}/\sqrt[n]{S_n}\}_{n\geq 1}$ is strictly decreasing to the limit $1$.
\end{theo}
Additionally, Z. W. Sun put forward a similar conjecture on the sequence $\{R_n\}_{n=0}^\infty.$

\begin{conj}\label{Conj: conjecture_R} The sequence $\{R_{n+1}/R_n\}_{n\geq 3}$ is strictly increasing to the limit $3+2\sqrt{2}$, and the sequence $\{\sqrt[n+1]{R_{n+1}}/\sqrt[n]{R_n}\}_{n\geq 5}$ is strictly decreasing to the limit $1$.
\end{conj}
Note that all progress and results (mentioned above) related to this subject can only be used to tackle with number theoretic or combinatorial sequence satisfying special expression or three-term recurrence relationship. However, one can not obtain a three-term recurrence for $S_n$ or $R_n$ only by using the Zeilberger's algorithm \cite{a=b,zeil} developed in \cite{chm,holo,fast}.
For example, one can easily acquire  four-term recurrences for $S_n$ and $R_n$ by using the holonomic method in \cite{holo} or the Zeilberger's algorithm \cite{zeil,a=b}, i.e.,
\begin{equation}\label{Rec:Seq-S}
 \begin{split}
 &9(n+1)^2S_n-(19n^2+74n+87)S_{n+1}+(n+3)(11n+29)S_{n+2}\\
 &-(n+3)^2S_{n+3}=0,
 \end{split}
 \end{equation}
 and
  \begin{equation}\label{rec-R}
 (n+3) R_{n+3}-(7 n+13) R_{n+2}+(7 n+15) R_{n+1}-(n+1)R_n=0.
 \end{equation}
However, Sun and Yang found a three-term recurrence relationship for $S_n$ by invoking a result in \cite{GL} and then proved Theorem \ref{Thm:thm-S} by establishing a new criterion for log-convexity and using a criterion for ratio log-concavity due to Chen, Guo and Wang \cite{cgw}.

In the present paper, by establishing a bounds for $R_{n+1}/R_n$ and using the interlacing method and a criterion(Theorem \ref{Thm:criterion-xia}) for log-concavity of the sequence of the form
$\{\sqrt[n]{z_n}\}_{n=1}^\infty$ in \cite{X}, we will completely solve Conjecture \ref{Conj: conjecture_R} on the condition that we do not know a three-term recurrence relationship of $R_n$.
\begin{theo}\emph{(\cite[Theorem 2.1]{X})}
\label{Thm:criterion-xia}
Let $\{z_n\}_{n=0}^\infty$ be a positive sequence. If there exist positive number $k_0$, positive integer $N_0$, and a function $f(n)$ such that $k_0<N_0^2+N_0+2$ and for $n\geq N_0$,
\begin{itemize}
\item[(i)] $0<f(n)<\frac{z_n}{z_{n-1}}<f(n+1);$
\item[(ii)] $\frac{f(n+1)}{f(n+3)}>1-\frac{k_0}{n^2+n+2};$
\item[(iii)]
$\left(1-\frac{k_0}{N_0^2+N_0+2}\right)^{N_0^2+N_0+2}
f^{2N_0}(N_0)>z_{N_0}^2;$
\end{itemize}
then the sequence $\{\sqrt[n]{z_n}\}_{n=N_0}^\infty$ is strictly log-concave.
\end{theo}
Our main result in this paper is as follows:
\begin{theo}
\label{Thm:confirm conj}
Conjecture \ref{Conj: conjecture_R} is true.
\end{theo}

In what follows, we will introduce the interlacing method in Section \ref{Sec:section-2}. In Section \ref{Sec:section-3}, a lower bound and an upper bound for $R_{n+1}/R_n$ will be established.
We will give and prove some properties related to the sequence $\{R_n\}_{n=0}^\infty$ in Section \ref{Sec:section-4} and prove Theorem \ref{Thm:confirm conj} therein.

 \section{The interlacing method}
 \label{Sec:section-2}
  The interlacing method can be found in \cite{WL}, yet it was formally considered as a method to solve logarithmic behavior of combinatorial sequences by Do\u{s}li\'{c} and Veljan \cite{DV}, in which it is also called sandwich method.

 To be self-contained in our paper. Let us give a simple introduction to the method.
 Suppose that $\{z_n\}_{n\geq 0}^\infty$ is a sequence of positive numbers. We define the sequence of consecutive quotients, i.e.,
 $$q_n=\frac{z_n}{z_{n-1}},\,\,\,n\geq 1.$$

 By the inequality in \eqref{Ineq:log-cv and log-cc}, the log-convexity or log-concavity of a sequence $\{z_n\}_{n\geq 0}$ is equivalent, respectively, to $q_n\leq q_{n+1}$ or $q_n\geq q_{n+1}$ for all $n\geq 1$. So it suffices to consider whether the sequence $\{q_n\}_{n\geq 1}$ decreases or increases, i.e., the ratio monotonicity.
 To prove $\{q_n\}_{n\geq 1}$ increases(resp. decreases), it is enough to
 find an increasing(resp. a decreasing) sequence $\{b_n\}_{n\geq 0}$ such that
 \begin{equation}\label{Ineq: interlacing mehtod}
 b_{n-1}\leq q_n\leq b_n\,\,(\text{resp.~}b_{n-1}\geq q_n\geq b_n)
 \end{equation}
  holds for all $n\geq 1$, or at least for all $n\geq N$ for some positive integer $N$. Clearly, this implies $q_n\leq q_{n+1}$(resp. $q_n\geq q_{n+1}$ ) since we have $\ldots\leq b_{n-1}\leq q_{n}\leq b_n\leq q_{n+1}\leq b_{n+1}\leq \ldots$(resp. $\ldots\geq b_{n-1}\geq q_{n}\geq b_n\geq q_{n+1}\geq b_{n+1}\geq \ldots$ ) by
  \eqref{Ineq: interlacing mehtod}. As a summary, we have the following proposition.
  \begin{prop}
  \label{Prop:interlacing method}
  Suppose that $\{z_n\}_{n\geq 0}$ is a sequence of positive numbers. Then for some positive integer $N$, the sequence $\{z_n\}_{n\geq N}$ is log-convex(resp. log-concave) if there exists an increasing(resp. a decreasing) sequence $\{b_n\}_{n\geq 0}$ such that
  \begin{equation*}
 b_{n-1}\leq q_n\leq b_n\,\,(\text{resp.~}b_{n-1}\geq q_n\geq b_n)
 \end{equation*}
 holds for $n\geq N+1$.
  \end{prop}
\section{Bounds for $R_{n+1}/R_n$}
 \label{Sec:section-3}
 In this section, a lemma on the bounds of  $R_{n+1}/R_n$ will be established. This lemma are very important for proof of our main results in the following sequel.

\begin{lemm}
\label{Lem:bounds-R}
Let
\begin{equation*}
\begin{split}
b_n&=3+2\sqrt{2}-\frac{3 (41\sqrt{2}+58)}{(14 \sqrt{2}+20) n}\\
&=\left(3-\frac{9}{2n}\right)+\sqrt{2}\left(2-\frac{3}{n}\right).
\end{split}
\end{equation*}
Then for $n\geq 3$, we have
\begin{equation*}
b_n<r_n<b_{n+1}.
\end{equation*}
\end{lemm}

\begin{table}
\centering
\caption{The first values for $b_n,r_n,b_{n+1}$}
\begin{tabular}{c|c|c|c|c|c|c|c}
\hline
\hline
  $r_n,b_n\backslash n$  & 3 & 4 & 5 & 6 & 7 & 8 & 9 \\
  \hline
  $b_{n+1}$& 3.64277 & 4.0799 & 4.37132 & 4.57948 & 4.7356 & 4.85702 & 4.95416\\
  \hline
  $r_n$ & 3.48 & 3.78161 & 4.1307 & 4.41575 &4.62573 & 4.78004 & 4.89728   \\
  \hline
  $b_n$ & 2.91421 & 3.64277 & 4.0799 & 4.37132 & 4.57948 & 4.7356 & 4.85702  \\
  \hline
\end{tabular}\label{table}
\end{table}
\begin{proof}
 Consider that the recurrence relationship \eqref{rec-R} implies that for $n\geq 0$
 \begin{equation}\label{ratio-rec}
\frac{R_{n+3}}{R_{n+2}}=
\frac{7n+13}{n+3}-\frac{7n+15}{n+3}\frac{R_{n+1}}{R_{n+2}}+\frac{n+1}{n+3}\frac{R_n}{R_{n+2}}.
 \end{equation}
Let $r_n=\frac{R_{n+1}}{R_n}$, the recurrence \eqref{ratio-rec} is equivalent to
\begin{equation}
r_{n+2}=\frac{7n+13}{n+3}-
\frac{7n+15}{n+3}\cdot\frac{1}{r_{n+1}}+\frac{n+1}{n+3}\cdot\frac{1}{r_nr_{n+1}}.
\end{equation}
First seeing that $b_n<r_n<b_{n+1}$ holds for $3\leq n\leq 8$ by  Table \ref{table}. So we proceed the proof by induction.
Suppose that $b_n<r_n<b_{n+1}$ for $n\leq k+1$, i.e.,
$$\cdots<b_k<r_k<b_{k+1}<r_{k+1}<b_{k+2}.$$

We first show that
\begin{equation}\label{ineq-1}
r_{k+2}-b_{k+3}=\frac{7k+13}{k+3}-
\frac{7k+15}{k+3}\cdot\frac{1}{r_{k+1}}+\frac{k+1}{k+3}\cdot\frac{1}{r_kr_{k+1}}-b_{k+3}<0.
\end{equation}
Note that $r_n,\,b_n$ are all positive numbers for $n\geq 3$, and by inductive hypothesis, we have
\begin{align*}
&\frac{7k+13}{k+3}-
\frac{7k+15}{k+3}\cdot\frac{1}{r_{k+1}}+\frac{k+1}{k+3}\cdot\frac{1}{r_kr_{k+1}}
-b_{k+3}\\
&<\frac{7k+13}{k+3}-
\frac{7k+15}{k+3}\cdot\frac{1}{b_{k+2}}+\frac{k+1}{k+3}\cdot\frac{1}{b_kb_{k+1}}
-b_{k+3}\\
&=\frac{(7k+13)b_kb_{k+1}b_{k+2}-(7k+15)b_kb_{k+1}+(k+1)b_{k+2}-(k+3)b_kb_{k+1}b_{k+2}b_{k+3}}
{(k+3)b_kb_{k+1}b_{k+2}}.
\end{align*}
The inequality \eqref{ineq-1} is equivalent to show that
\begin{equation}
\begin{split}
&(7k+13)b_kb_{k+1}b_{k+2}-(7k+15)b_kb_{k+1}+(k+1)b_{k+2}-(k+3)b_kb_{k+1}b_{k+2}b_{k+3}\\
&=\frac{3 \left(2 k \left(-4 \left(179+127 \sqrt{2}\right) k+1110 \sqrt{2}+1573\right)-844 \sqrt{2}-1197\right)}{16 k (k+1) (k+2)}\\
&=\frac{-24 \left(179+127 \sqrt{2}\right) k^2+3 \left(3146+2220 \sqrt{2}\right) k-3 \left(1197+844 \sqrt{2}\right)}{16 k (1 + k) (2 + k)}\\
&<0.
\end{split}\label{ineq-2}
\end{equation}
Let $f(x)=ax^2+bx+c$, where $a=-24 \left(179+127 \sqrt{2}\right) $, $b=3 \left(3146+2220 \sqrt{2}\right)$ and $c=-3 \left(1197+844 \sqrt{2}\right).$
Note that the symmetric axis is $x=-\frac{b}{2a}=\frac{3146+2220 \sqrt{2}}{16 \left(179+127 \sqrt{2}\right)}\approx 1.09549$, so for $k\geq 3$, we have that
$f(k)\leq f(3)=-3 (4647 + 3328 \sqrt{2})<0$, which implies the inequality \eqref{ineq-2}. Thus we complete the proof of $r_{n+2}<b_{n+3}$ according to the principle of inductive argument.

Now we have to prove $r_{k+2}>b_{k+2}$. Similarly, according to inductive hypothesis, it suffices to show that
\begin{equation}\label{ineq-3}
\begin{split}
&r_{k+2}-b_{k+2}\\
&=\frac{7k+13}{k+3}-
\frac{7k+15}{k+3}\cdot\frac{1}{r_{k+1}}+\frac{k+1}{k+3}\cdot\frac{1}{r_kr_{k+1}}
-b_{k+2}\\
&>\frac{7k+13}{k+3}-
\frac{7k+15}{k+3}\cdot\frac{1}{b_{k+1}}+\frac{k+1}{k+3}\cdot\frac{1}{b_{k+1}b_{k+2}}
-b_{k+2}\\
&=\frac{(7k+13)b_{k+1}b_{k+2}-(7k+15)b_{k+2}+(k+1)-(k+3)b_{k+1}b_{k+2}^2}
{(k+3)b_{k+1}b_{k+2}}\\
&>0.
\end{split}
\end{equation}
So, to prove $r_{k+2}>b_{k+2}$ is equivalent to show the \eqref{ineq-3} is valid for $k$ greater than some integer $k_0.$

Consider that
\begin{align*}
&(7k+13)b_{k+1}b_{k+2}-(7k+15)b_{k+2}+(k+1)-(k+3)b_{k+1}b_{k+2}^2\\
&=\frac{3 \left(k \left(4 \left(47+33 \sqrt{2}\right) k-370 \sqrt{2}-519\right)-218 \sqrt{2}-305\right)}{8 (k+1) (k+2)^2}\\
&=\frac{\left(564+396 \sqrt{2}\right) k^2-\left(1557+1110 \sqrt{2}\right) k-654 \sqrt{2}-915}{8 (k+1) (k+2)^2}.
\end{align*}
If we let $g(x)=a_1x^2+b_1x+c_1$, where $a_1=\left(564+396 \sqrt{2}\right) ,\,b_1=-\left(1557+1110 \sqrt{2}\right)$ and $c_1=-654 \sqrt{2}-915$, then the symmetric axis of $g(x)$ is $x=-\frac{b_1}{2a_1}=-\frac{1557+1110 \sqrt{2}}{2 \left(915+654 \sqrt{2}\right)}\approx -0.849716.$ Hence, for $k\geq 4$, we have
$g(k)\geq g(4)=9 \left(209+138 \sqrt{2}\right)>0,$ which implies that the inequality \eqref{ineq-3}, i.e., $r_{k+2}\geq b_{k+2}.$

According the above analysis and inductive argument, for all $n\geq 3$,
we have $b_n<r_n<b_{n+1}.$
\qed
\end{proof}
We are now in a position to consider logarithmic behavior of the sequences $\{R_n\}_{n=0}^\infty$.
 \section{Log-behavior of the sequence $\{R_n\}_{n=0}^\infty$}

 By Lemma \ref{Lem:bounds-R} and the sequence $\{b_n\}_{n=1}^\infty$ is strictly increasing, we can first obtain the log-convexity of $\{R_n\}_{n=0}^\infty$.
 \label{Sec:section-4}
\begin{theo}
\label{Thm:R-lcv-ratio-increasing}
The sequence $\{R_n\}_{n=4}^\infty$ is strictly log-convex. Equivalently, the sequence
$\{R_{n+1}/R_n\}_{n=3}^\infty$ is strictly increasing.
\end{theo}
\begin{proof}
Consider that $R_3^2-R_2R_1=25^2-7\cdot 87=16>0$, and by Lemma \ref{Lem:bounds-R},
we have that for $n\geq 3$,
$$\ldots<b_n<r_n=\frac{R_{n+1}}{R_n}<b_{n+1}<r_{n+1}<b_{n+2}<\ldots.$$
This arrives at the sequence $\{r_n\}_{n=3}^\infty$ is strictly increasing.
Thus it implies that $\{R_n\}_{n=4}^\infty$ is log-convex by Proposition \ref{Prop:interlacing method}.
\qed
\end{proof}

What's more, note that
\begin{align*}
\lim_{n\rightarrow\infty}b_n=3+2\sqrt{2},
\,\,\,\lim_{n\rightarrow\infty}
b_{n+1}=3+2\sqrt{2}.
\end{align*}
It follows easily the following result.
\begin{theo}
\label{Thm:limit-r-n}
\begin{equation*}
\lim_{n\rightarrow\infty}\frac{R_{n+1}}{R_n}=3+2\sqrt{2}. \end{equation*}
\end{theo}

\begin{coro}\label{Cor:R_kaifang-increase}
The sequence $\{\sqrt[n]{R_n}\}_{n=1}^\infty$ is strictly increasing. Moreover,
\begin{equation}\label{Lim:kaifang-limit}
\lim_{n\rightarrow\infty}\sqrt[n]{R_n}=3+2\sqrt{2}.
\end{equation}
\end{coro}
\begin{proof}
By Theorem \ref{Thm:R-lcv-ratio-increasing}, we have
\begin{equation*}
\frac{R_{n+1}}{R_n}>\frac{R_n}{R_{n-1}},\,\,\mbox{for}\,\,n\geq 3.
\end{equation*}
With the fact $R_1=1$, we deduce that for $n\geq 1$,
\begin{equation}\label{Inequality-1}
R_n=\frac{R_2}{R_1}\cdot \frac{R_3}{R_2}\left[\cdot R_1 \cdot\frac{R_4}{R_3} \cdots\frac{R_n}{R_{n-1}}\right]<R_3\left(\frac{R_{n+1}}{R_n}\right)^{n-2}.\\
\end{equation}
For $n\geq 11$, we have
\begin{equation}\label{Inequality-2}
\frac{R_{n+1}}{R_n}\geq\frac{R_{12}}{R_{11}}
=\frac{16421831}{3242377}>5=\sqrt{R_3}.
\end{equation}
Combining these inequalities in \eqref{Inequality-1} and \eqref{Inequality-2}, it follows that
\begin{equation*}
R_n^{n+1}<R_{n+1}^n, \text{~for~} n\geq 11.
\end{equation*}
This is equivalent to
\begin{equation*}
(R_n^{n+1})^{\frac{1}{n(n+1)}}
<(R_{n+1}^n)^{\frac{1}{n(n+1)}},\text{~for~} n\geq 11.
\end{equation*}
That is,
\begin{equation*}
\sqrt[n]{R_n}<\sqrt[n+1]{R_{n+1}}, \text{~for~}n\geq 11.
\end{equation*}
For $1\leq n\leq 10$, one can simply prove $R_n^{n+1}<R_{n+1}^n$ by computing the value of $R_n^{n+1}-R_{n+1}^n$. Here are some examples,
\begin{align*}
&R_1^{2}-R_{2}=1-7=-6;\\
&R_2^{3}-R_{3}^2=343-625=-282;\\
&R_3^{4}-R_{4}^3=390625-658503=-267878; \\
&R_4^{5}-R_{5}^4=4984209207-1268163904241521=-6731904874;\\
&R_5^{6}-R_{6}^5= 1268163904241521-1268163904241521=-3367343548629278.
\end{align*}
Moreover, recall that for a real sequence $\{z_n\}_{n=1}^\infty$ with positive numbers, it was shown that
\begin{equation}\label{in-1}
\lim_{n\rightarrow\infty}\inf\frac{z_{n+1}}{z_n}\leq \lim_{n\rightarrow\infty}\inf\sqrt[n]{z_n}, \end{equation}
and
\begin{equation}\label{in-2}
\lim_{n\rightarrow\infty}\sup\sqrt[n]{z_n}\leq \lim_{n\rightarrow\infty}\sup\frac{z_{n+1}}{z_n}, \end{equation}
see Rudin \cite[\S 3.37]{R}. The inequalities in \eqref{in-1} and \eqref{in-2} implies that
$$\lim_{n\rightarrow\infty}\sqrt[n]{z_n}=
\lim_{n\rightarrow\infty}\frac{z_n}{z_{n-1}}$$ if $\lim_{n\rightarrow\infty}\frac{z_n}{z_{n-1}}$ exists. By Theorem \ref{Thm:limit-r-n}, it follows \eqref{Lim:kaifang-limit}.

This completes the proof.
\qed
\end{proof}

\begin{theo}\label{Thm:limit-ratio-kaifang}
\begin{equation*}
\lim_{n\rightarrow \infty}\frac{\sqrt[n+1]{R_{n+1}}}{\sqrt[n]{R_n}}
=1.
\end{equation*}
\end{theo}
\begin{proof}
Not that for $n\geq 3$
\begin{equation*}
R_{n+1}=R_3\prod_{k=3}^nr_k.
\end{equation*}
Hence by Lemma \ref{Lem:bounds-R}, it follows that
\begin{equation*}
R_3\prod_{k=3}^nb_k<R_{n+1}<R_3\prod_{k=3}^nb_{k+1}.
\end{equation*}

Consider that
\begin{equation*}
\begin{split}
\log\left(\frac{\sqrt[n+1]{R_{n+1}}}{\sqrt[n]{R_n}}\right)
&=\frac{\log\left.R_{n+1}\right.}{n+1}-\frac{\log\left.R_{n}\right.}{n}\\
&<\frac{\log\left(R_3\prod_{k=3}^nb_{k+1}\right)}{n+1}-\frac{\log\left(R_3\prod_{k=3}^{n-1}b_k\right)}{n}\\
&=\frac{\log R_3+\sum_{k=3}^n\log b_{k+1}}{n+1}-\frac{\log R_3 +\sum_{k=3}^{n-1}\log b_k}{n}
\end{split}
\end{equation*}
and
\begin{equation*}
\begin{split}
\log\left(\frac{\sqrt[n+1]{R_{n+1}}}{\sqrt[n]{R_n}}\right)
&=\frac{\log\left.R_{n+1}\right.}{n+1}-\frac{\log\left.R_{n}\right.}{n}\\
&>\frac{\log\left(R_3\prod_{k=3}^nb_k\right)}{n+1}-\frac{\log\left(R_3\prod_{k=3}^{n-1}b_{k+1}\right)}{n}\\
&=\frac{\log R_3+\sum_{k=3}^n\log b_k}{n+1}-\frac{\log R_3 +\sum_{k=3}^{n-1}\log b_{k+1}}{n}.
\end{split}
\end{equation*}
By using mathematical software {\tt{Mathematica 10.0}}, we can obtain that
\begin{align*}
&\lim_{n\rightarrow \infty}\left(\frac{\log R_3+\sum_{k=3}^n\log b_{k+1}}{n+1}-\frac{\log R_3 +\sum_{k=3}^{n-1}\log b_k}{n}\right)=0,\\
&\lim_{n\rightarrow \infty}\left(\frac{\log R_3+\sum_{k=3}^n\log b_k}{n+1}-\frac{\log R_3 +\sum_{k=3}^{n-1}\log b_{k+1}}{n}\right)=0.
\end{align*}
The above two limits force that
$$\lim_{n\rightarrow \infty}\log\left(\frac{\sqrt[n+1]{R_{n+1}}}{\sqrt[n]{R_n}}\right)=0,$$
which implies
$$\lim_{n\rightarrow \infty}\frac{\sqrt[n+1]{R_{n+1}}}{\sqrt[n]{R_n}}=1.$$
\qed
\end{proof}
\begin{theo}\label{Thm:kaifang-lcc}
The sequence $\{\sqrt[n]{R_n}\}_{n=5}^\infty$ is strictly log-concave. Equivalently, the sequence $\{\frac{\sqrt[n+1]{R_{n+1}}}{\sqrt[n]{R_n}}\}_{n=5}^\infty$ is strictly decreasing.
\end{theo}
\begin{proof}
We will prove it by Theorem \ref{Thm:criterion-xia}. To keep notation in Theorem \ref{Thm:criterion-xia}, we let
$f(n)=b_{n-1}=\sqrt{2} \left(2-\frac{3}{n-1}\right)-\frac{9}{2 (n-1)}+3$.
First, by Lemma \ref{Lem:bounds-R}, we know that
\begin{equation*}
0<f(n)<\frac{R_n}{R_{n-1}}<f(n+1),\text{~for~}n\geq 5.
\end{equation*}
Note that
$$\frac{f(n+1)}{f(n+3)}=\frac{12}{2 n+1}-\frac{6}{n}+1$$
and
\begin{align*}
&\left(\frac{12}{2 n+1}-\frac{6}{n}+1\right)-
\left(1-\frac{4}{n^2+n+2}\right)\\
&\quad=\frac{2 (n-3) (n+2)}{n (2 n+1) \left(n^2+n+2\right)}\\
&\quad >0, \text{~for~}n\geq 4.
\end{align*}
So, taking $k_0=4$, the condition (ii) in Theorem \ref{Thm:criterion-xia} is satisfied.

Moreover, consider that
\begin{align*}
\left(1-\frac{4}{8^2+8+2}\right)^{8^2+8+2}
f^{16}(8)-R_{8}^2=-1.5798\times 10^8
\end{align*}
and
\begin{align*}
\left(1-\frac{4}{9^2+9+2}\right)^{9^2+9+2}
f^{18}(9)-R_{9}^2=6.41905\times 10^9.
\end{align*}
Therefore, let $N_0=9, k_0=4, f(n)=b(n-1)$, all the conditions (i), (ii) and (iii) in Theorem \ref{Thm:criterion-xia} can be satisfied. This implies that the sequence
$\{\sqrt[n]{R_n}\}_{n=9}^\infty$ is strictly log-concave, which is equivalent to $\{\frac{\sqrt[n+1]{R_{n+1}}}{\sqrt[n]{R_n}}\}_{n=9}^\infty$
is strictly decreasing by Proposition \ref{Prop:interlacing method}.

However, one can verify that for $5\leq n\leq 8$,
\begin{align*}
\frac{\sqrt[n+1]{R_{n+1}}}{\sqrt[n]{R_n}}>\frac{\sqrt[n+2]{R_{n+2}}}{\sqrt[n+1]{R_{n+1}}},
\end{align*}
since
\begin{align*}
&\frac{\sqrt[6]{R_{6}}}{\sqrt[5]{R_5}}
-\frac{\sqrt[7]{R_{7}}}{\sqrt[6]{R_{6}}}\approx0.00293164,&
&\frac{\sqrt[7]{R_{7}}}{\sqrt[6]{R_6}}
-\frac{\sqrt[8]{R_{8}}}{\sqrt[7]{R_{7}}}\approx0.00445875,\\
&\frac{\sqrt[8]{R_{8}}}{\sqrt[7]{R_7}}
-\frac{\sqrt[9]{R_{9}}}{\sqrt[8]{R_{8}}}\approx0.00452784,&
&\frac{\sqrt[9]{R_{9}}}{\sqrt[8]{R_8}}
-\frac{\sqrt[10]{R_{10}}}{\sqrt[9]{R_{9}}}\approx0.00404051.
\end{align*}
Thus this completes the proof.
\qed
\end{proof}
\begin{remark}
We prove Theorem \ref{Thm:kaifang-lcc} by invoking Theorem \ref{Thm:criterion-xia} due to Xia \cite{X}. Actually, this method can also be use to prove the log-concavity of the sequence $\{\sqrt[n]{S_n}\}_{n=1}^\infty$ in Theorem \ref{Thm:thm-S}, which was proved by using the method in \cite{cgw}.
\end{remark}

Now we are ready to proof of Theorem \ref{Thm:confirm conj}.

\emph{Proof of Theorem \ref{Thm:confirm conj}.}
By Theorem \ref{Thm:R-lcv-ratio-increasing} and \ref{Thm:limit-r-n}, we confirm the first part of Conjecture \ref{Conj: conjecture_R}.
Moreover, Theorem \ref{Thm:kaifang-lcc} and \ref{Thm:limit-ratio-kaifang} imply the second part of Conjecture \ref{Conj: conjecture_R}.

We are done.
\qed
\begin{conj}
The sequence $\{r_n\}_{n\geq 4}$ is log-concave, i.e., $R_n$ is ratio log-concave for $n\geq 4.$
\end{conj}

\begin{conj}
The sequence $\{R_n^2-R_{n+1}R_{n-1}\}_{n\geq 6}$ is $\infty$-log-concave.
\end{conj}

\vspace{.3cm}

\noindent{\bf Acknowledgments.} This work was
supported by the National Science Foundation of China.

\end{document}